\documentclass[bibtex,12pt,en]{elegantpaper}
\usepackage{extarrows}
\usepackage{mathrsfs}
\numberwithin{equation}{section}
\allowdisplaybreaks[4]
\everymath{\displaystyle}

\title{Liouville theorem of the subcritical biharmonic equation on complete manifolds}

\author{Xi-Nan Ma \and Tian Wu \and Wangzhe Wu}

\begin{document}

\date{}
\maketitle

\renewcommand{\thefootnote}{\fnsymbol{footnote}}

%\footnotetext[1]{The corresponding author.}

\begin{abstract}

    In this paper, we study the subcritical biharmonic equation
    \[\Delta ^2 u=u^\alpha\]
    on a complete, connected, and non-compact Riemannian manifold \((M^n,g)\) with nonnegative Ricci curvature. Using the method of invariant tensors, we derive a differential identity to obtain a Liouville theorem, i.e., there is no positive \(C^4\) solution if \(n\geqslant5\) and \(1<\alpha<\frac{n+4}{n-4}\). We establish a crucial second-order derivative estimate, which is established via Bernstein's technique and the continuity method.
\end{abstract}

\tableofcontents
\section{Introduction}\label{sec:intro}

Let \(n\geqslant5\) and \((M^n,g)\) be a complete, connected, and non-compact Riemannian manifold of dimension \(n\) with nonnegative Ricci curvature. Let \(\nabla\) and \(\Delta\) be the Levi-Civita connection and the Laplace-Beltrami operator w.r.t. the metric \(g\). Denote the Ricci curvature as \(\mathrm{Ric}\).

We investigate positive solutions to the fourth-order subcritical equation
\begin{equation}\label{eq}
    \Delta^2u=u^\alpha\quad\text{in }(M^n,g).
\end{equation}
In this paper, we focus on the subcritical case \(1<\alpha<\frac{n+4}{n-4}\). With the help of the invariant tensor technique, we establish a crucial differential identity, implying the following Liouville theorem.
\begin{theorem}\label{thm:subcritical}
    Let \(n\geqslant5\) and \((M^n,g)\) be a complete, connected, and non-compact Riemannian manifold of dimension \(n\) with nonnegative Ricci curvature. If \(1<\alpha<\frac{n+4}{n-4}\) and \(\Delta u\leqslant0\), then equation \eqref{eq} admits no positive \(C^4\) solution.
\end{theorem}

Related second-order semilinear equations in complete, connected, and non-compact Riemannian manifolds with nonnegative Ricci curvature have been widely studied. Gidas-Spruck \cite{GS81} proved non-existence of positive solution to \(-\Delta u=u^{\alpha}\) in subcritical case \(1<\alpha<\frac{n+2}{n-2}\) with \(n\geqslant 3\). The critical case \(\alpha=\frac{n+2}{n-2}\) is more difficult. Caffarelli-Gidas-Spruck \cite{CGS89} (also see Chen-Li \cite{CL91}) classified positive solutions via the moving plane method in \(\mathbb R^n\). The rigidity result in the critical case, that is, the manifold must be \(\mathbb R^n\), was investigated under some assumptions of decay at infinity or integrable condition, such as Fogagnolo-Malchiodi-Mazzieri \cite{FMM23}, Catino-Monticelli \cite{CM22}. Recently, to remove assumptions, some works succeeded in lower dimensions, such as Catino-Monticelli-Roncoroni \cite{CMR23}, Ou \cite{Ou25}, V\'etois \cite{Vet24}, Sun-Wang \cite{SW25}.

Let's focus on the fourth-order case. Lin \cite{Lin98} classified solutions to \(\Delta^2u=\mathrm e^{4u}\) in \(\mathbb R^4\) and positive solutions to \(\Delta^2u=u^\alpha\) in \(\mathbb R^n\) with \(n\geqslant5\) and \(1<\alpha\leqslant\frac{n+4}{n-4}\) by the moving plane method. However, moving planes doesn't work in general manifolds. It's noteworthy that corresponding problems on closed manifolds with a positive lower Ricci curvature bound have been studied recently. We recommend readers to see V\'etois \cite{Vet24}, Li-Wei \cite{LW25}, Case \cite{Cas24}, Ma-Wu-Zhou \cite{MWZ25a,MWZ25b}.

Hence, in Theorem \ref{thm:subcritical}, we extend Lin's Liouville theorem in the subcritical case to complete non-compact manifolds. Unlike backgrounds from conformal geometry in \cite{Vet24,LW25}, we establish differential identities using the invariant tensor technique. This technique was introduced by Ma-Wu \cite{MW24} to study second-order semi-linear equations on manifolds. Recently, Ma-Wu-Zhou \cite{MWZ25a} extended Vetois's and Li-Wei's result on closed manifolds via this technique.

In local coordinates, we denote the metric tensor as \(g_{ij}\), and the Ricci curvature tensor as \(R_{ij}\). Denote the inner product w.r.t. \(g\) as \(\langle\cdot,\cdot\rangle\). We employ the Einstein summation convention to raise and lower the indices using the metric tensor \(g_{ij}\) and its inverse \(g^{ij}\). All covariant derivatives are denoted with a comma, such as \(_{,i}\), except for acting on our solution \(u\).

Throughout this paper, we assume \(R>1\). The constant \(\varepsilon\) and \(C\) denote positive constants independent of \(R\). Besides, \(\varepsilon\) is always sufficiently small. Denote the geodesic ball centered at \(o\) with radius \(R\) as \(B_R\). Let \(\eta\) be a smooth cutoff function supported on \(\overline{B_{2R}}\) satisfying
\[0\leqslant\eta\leqslant 1, \quad\eta\equiv 1,\quad\text{in } B_R,\quad\text{and}\quad |\nabla^k\eta|\leqslant CR^{-k}\text{ for }k=1,2.\]

In Section \ref{sec:estimate}, we establish the prior estimate of second derivatives via Berstein's technique and the continuity method. In Section \ref{sec:identity}, the crucial differential identity is obtained by the invariant tensor technique. In Section \ref{sec:subcritical}, using the differential identity, we establish integral estimates to prove Theorem \ref{thm:subcritical}.
\section{An estimate of second order derivatives}\label{sec:estimate}

In this section, we establish a prior estimate of second-order derivatives for a positive function \(u\) satisfying \(\Delta ^2u\geqslant0\). This estimate is crucial for verifying the positivity of the differential identity.

By iteration arguments, Fazly-Wei-Xu \cite{FWX15} deduced the estimate
\[\Delta u+\frac{2}{n-4}\frac{|\nabla u|^2}{u}\sqrt{\frac{2}{p+1-c_n}}|x|^{\frac a 2}u^{\frac{p+1}{2}}\leqslant0,\quad c_n=\frac{8}{n(n-4)}\]
for the fourth-order H\'enon equation \(\Delta^2u=|x|^au^p\) in \(\mathbb R^n\). Ng\^o-Nguyen-Phan \cite{NNP18} studied the estimate of the equation with negative exponent, namely \(\Delta^2u=-u^{-q}\) in \(\mathbb R^n\), and obtained that
\[-\Delta u+\alpha\frac{|\nabla u|^2}{u}+\beta u^{-\frac{q-1}{2}}\]
with some conditions of \((\alpha,\beta)\). Without concerning the power term of \(u\), we obtain a simple way to establish the estimate on complete manifolds by Bernstein's technique and the continuity method. 

Set \(Z_a=u^{-1}\Delta u+ au^{-2}|\nabla u|^2\), \(a\geqslant 0\). Then
\[\nabla Z_a=-u^{-2}\Delta u\nabla u+u^{-1}\nabla \Delta u-2au^{-3}|\nabla u|^2\nabla u+2au^{-2}\nabla^2u\cdot\nabla u.\]
\begin{align*}
    &u^{2a-2}\operatorname{div}(u^{2-2a}\nabla Z_a)\\
    =~&\frac{\Delta^2 u}{u}-\frac{(\Delta u)^2}{u^2}+2a(1+2a)\frac{|\nabla u|^4}{u^4}+2a\operatorname{Ric}(\frac{\nabla u}{u},\frac{\nabla u}{u})\\
    &-4a(1+a)\frac{\langle\nabla^2 u,\nabla u\otimes\nabla u\rangle}{u^3}+2a\frac{|\nabla^2u|^2}{u^2}\\
    =~&2a\Big|\frac{\nabla^2u}{u}-(1+a)\frac{\nabla u\otimes\nabla u}{u^2}-\frac 1 n\Big(\frac{\Delta u}{u}-(1+a)\frac{|\nabla u|^2}{u^2}\Big)g\Big|^2+\frac{\Delta^2 u}{u}+2a\operatorname{Ric}(\frac{\nabla u}{u},\frac{\nabla u}{u})\\
    &+(\frac 2 n a-1)\frac{(\Delta u)^2}{u^2}-\frac 4 n a(1+a)\frac{\Delta u|\nabla u|^2}{u^3}+\frac 2 n a[(1+a)^2-na^2]\frac{|\nabla u|^4}{u^4}.
\end{align*}
By non-negativity of Ricci curvature, and \(\Delta^2u\geqslant0\), we obtain
\begin{align}
    \begin{split}\label{est:1}
        &u^{2a-2}\operatorname{div}(u^{2-2a}\nabla Z_a)\\
        \geqslant~&(\frac 2 n a-1)\Delta uZ_a-\frac a n(6a-n+4)\frac{|\nabla u|^2}{u^2}Z_a+\frac a n(1+2a)[2-(n-4)a]\frac{|\nabla u|^4}{u^4}.        
    \end{split}
\end{align}

The positivity of \eqref{est:1} ensures an estimate on \(Z_a\).

\begin{lemma}\label{lem:est}
 Assume that \(Z_a\eta^4\) attains its maximum at \(x_0\in B_{2R}\). If
 \begin{equation}\label{est:2}
        \operatorname{div}(u^{2-2a}\nabla Z_a)\geqslant\delta u^{2-2a}\Big(Z_a^2+\frac{|\nabla u|^4}{u^4}\Big).
 \end{equation}
    holds at \(x_0\) for some \(\delta>0\) depending on \(n,a\), then \(Z_a\leqslant0\).
\end{lemma}
 
\begin{proof}
 At \(x_0\), we have \(\eta^4\nabla Z_a+4Z_a\eta^3\nabla \eta=\nabla(Z_a\eta^4)=0\), and
 \begin{align*}
  0 &\geqslant\Delta(Z_a\eta^4)=u^{2a-2}\operatorname{div}[u^{2-2a}\nabla(Z_a\eta^4)]\\
  &=u^{2a-2}\operatorname{div}(u^{2-2a}\nabla Z_a)\eta^4-32Z_a\eta^2|\nabla \eta|^2+8(1-a)u^{-1}Z_a\eta^3\langle\nabla u,\nabla\eta\rangle+Z_a\Delta\eta^4.
 \end{align*}
    By \eqref{est:2} and Young's inequality, we obtain
 \[\delta\Big(Z_a^2+\frac{|\nabla u|^4}{u^4}\Big)\eta^4\leqslant CR^{-2}Z_a\eta^2+CR^{-1}u^{-1}|\nabla u|Z_a\eta^3\leqslant\frac\delta 2\Big(Z_a^2+\frac{|\nabla u|^4}{u^4}\Big)+CR^{-4}.\]
 Thus, \(Z_a\eta^4\leqslant CR^{-2}\) holds at \(x_0\). By the maximization of \(x_0\), we obtain \(Z_a\leqslant CR^{-2}\) in \(B_R\). The proof is finished by letting \(R\to\infty\).
\end{proof}

The target estimate relies on the continuity method, which calls for \(\Delta u\leqslant0\) priorly.
 
\begin{proposition}\label{prop:est}
    Let \((M^n,g)\) be a complete, connected, and non-compact Riemannian manifold of dimension \(n\) with non-negative Ricci curvature, and \(u\) be a positive function satisfying \(\Delta^2u\geqslant0\). If \(\Delta u\leqslant0\), then \(\Delta u+\frac{2}{n-4}\frac{|\nabla u|^2}{u}\leqslant0\).
\end{proposition}
 
\begin{proof}
 Define \(T=\Big\{a\in[0,\frac{2}{n-4}]\mid Z_a \leqslant 0\text{ in }M\Big\}\), then \(0\in T\) and \(T\) is closed.
    
    We only need to show that \(T\) is also open. Take \(a_0\in T\), and \(a\in(a_0,a_0+\varepsilon)\cap[0,\frac{2}{n-4}]\). Assume that \(Z_a\eta^4\) attains its non-negative maximum at \(x_0\in B_{2R}\), then
 \begin{equation}\label{est:3}
     0\leqslant Z_a=Z_{a_0}+(a-a_0)u^{-2}|\nabla u|^2\leqslant(a-a_0)u^{-2}|\nabla u|^2\leqslant\varepsilon u^{-2}|\nabla u|^2
 \end{equation}
    holds at \(x_0\). By \(\frac 2 na-1<0\), \(0<a<\frac{2}{n-4}\) and \eqref{est:3}, the inequality \eqref{est:1} yields
    \begin{equation}\label{est:4}
        u^{2a-2}\operatorname{div}(u^{2-2a}\nabla Z_a)\geqslant-C\frac{|\nabla u|^2}{u^2}Z_a+\frac a n(1+2a)[2-(n-4)a]\frac{|\nabla u|^4}{u^4}\geqslant2\delta\frac{|\nabla u|^4}{u^4},
    \end{equation}
    at \(x_0\), where \(\delta>0\) by taking \(\varepsilon>0\) small enough. By \eqref{est:3},
    \[Z_a^2=\frac{\Delta u}{u}Z_a+a\frac{|\nabla u|^2}{u^2}Z_a\leqslant a\frac{|\nabla u|^2}{u^2}Z_a\leqslant a\varepsilon\frac{|\nabla u|^4}{u^4}\]
    holds at \(x_0\). Thus, the condition \eqref{est:2} is verified by \eqref{est:4}, and \(a\in T\) by Lemma \ref{lem:est}.
\end{proof}
\section{Invariant tensors and the differential identity}\label{sec:identity}

In this section, we use the invariant tensor technique to establish differential identities.

Define the trace-free tensor
\[E_{ij}=u_{ij}+b\frac{u_i u_j}{u}-\frac{1}{n}\Big(\Delta u+b\frac{|\nabla u|^2}{u}\Big)g_{ij}\]
where \(b\) is a constant to be determined later. We trace the covariant divergence of \(E_{ij}\) to get
\[{E_{ij,}}^i=\frac{n-1}{n} (\Delta u)_{,j}+R_{ij} u^i +b\frac{\Delta u}{u} u_j+\frac{n-2}{n} b \frac{u_{ij} u^i}{u}-\frac{n-1}{n} b\frac{|\nabla u|^2}{u^2} u_j.\]
Plugging \(u_{ij}= E_{ij}-b\frac{u_i u_j}{u}+\frac{1}{n}\Big(\Delta u+b\frac{|\nabla u|^2}{u}\Big)g_{ij}\) into the above formula yields
\[{E_{ij,}}^i=\frac{n-2}{n}b E_j +\frac{n-1}{n}F_j+R_{ij} u^i,\]
where \(E_j:=\frac{E_{ij}u^i}{u}\), and \(F_j=(\Delta u)_{,j}+\frac{n+2}{n}b\frac{\Delta u}{u}u_j-b(1+\frac{n-2}{n}b)\frac{|\nabla u|^2}{u^2}u_j\). Thus
\begin{align*}
    {F_{i,}}^i=&~\Delta^2 u+\frac{n+2}{n} b \Big(\frac{{(\Delta u)_,}^i}{u} u_i+\frac{(\Delta u)^2}{u}-\frac{\Delta u |\nabla u|^2}{u^2}\Big)\\
    &-b(1+\frac{n-2}{n}b)\Big(2\frac{u^{ij}u_iu_j}{u^2}+\frac{\Delta u |\nabla u|^2}{u^2}-2\frac{|\nabla u|^4}{u^3}\Big).
\end{align*}
Likewise, we replace \(u^{ij}\) and \({(\Delta u)_,}^i\) by \(E^{ij}\) and \(F^i\) respectively. It follows that \[{F_{i,}}^i=-2b(1+\frac{n-2}{n}b)E+\frac{n+2}{n}bF+G,\]
where \(E:=\frac{E_iu^i}{u}\), \(F:=\frac{F_iu^i}{u}\), and 
\[G=\Delta^2 u +\frac{n+2}{n}b\frac{(\Delta u)^2}{u} -\frac{2(n+2)}{n}b(1+b)\frac{\Delta u|\nabla u|^2}{u^2}+b(3b+2)(1+\frac{n-2}{n}b)\frac{|\nabla u|^4}{u^3}.\]
To compute \({G_,}^i\), we differentiate the equation \eqref{eq}, implying that \({(\Delta^2 u)_,}^i=\alpha\frac{\Delta^2u}{u}u^i\). Thus
\begin{align*}
    {G_,}^i=~&\alpha\frac{\Delta^2u}{u}u^i+\frac{n+2}{n}b\Big(2\frac{\Delta u{(\Delta u)_,}^i}{u}-\frac{(\Delta u)^2}{u^2} u^i\Big)\\
    &-\frac{2(n+2)}{n}b(1+b)\Big( \frac{|\nabla u|^2{(\Delta u)_,}^i}{u^2}+2\frac{\Delta u}{u^2}u^{ij}u_j-2\frac{\Delta u|\nabla u|^2}{u^3}u^i\Big)\\
    &+b(3b+2)(1+\frac{n-2}{n}b)\Big(4\frac{|\nabla u|^2}{u^3}u^{ij}u_j-3\frac{|\nabla u|^4}{u^4} u^i\Big).
\end{align*}
Replacing the terms \(u^{ij}\), \({(\Delta u)_,}^i\) and \(\Delta^2 u\) by \(E^{ij}\), \(F^i\), and \(G\) respectively yields
\begin{align*}
    {G_,}^i=&~4b\Big((3b+2)(1+\frac{n-2}{n}b)\frac{|\nabla u|^2}{u^2}-\frac{n+2}{n}(1+b)\frac{\Delta u}{u}\Big)E^i+\alpha\frac{G}{u}u^i\\
    &+\frac{2(n+2)}{n}b\Big(\frac{\Delta u}{u}-(1+b)\frac{|\nabla u|^2}{u^2}\Big)F^i-\frac{n+2}{n}b\Big(\frac{n+4}{n}(1+2b)+\alpha\Big)\frac{(\Delta u)^2}{u^2}u^i\\
    &+\frac{2b}{n}\Big(2(n+4)(1+2b)(1+\frac{n-1}{n}b)+(n+2)(1+b)\alpha\Big)\frac{\Delta u|\nabla u|^2}{u^3}u^i\\
    &-b(1+\frac{n-2}{n}b)\Big((1+2b)(6+\frac{7n-4}{n}b)+(2+3b)\alpha\Big)\frac{|\nabla u|^4}{u^4}u^i.
\end{align*}

For eliminating the term \(\frac{(\Delta u)^2}{u^2}u^i\) in \({G_,}^i\), we take \(b=-\frac 1 2(1+\frac{n\alpha}{n+4})\) in above formulae. Thus, the following proposition states relationships between invariant tensors.

\begin{proposition}\label{prop:invariant}
    Invariant tensors are taken as
    \[E_{ij}=u_{ij}-\frac 1 2(1+\frac{n\alpha}{n+4})\frac{u_i u_j}{u}-\frac{1}{n}\Big(\Delta u-\frac 1 2(1+\frac{n\alpha}{n+4})\frac{|\nabla u|^2}{u}\Big)g_{ij},\]
    \[F_j=(\Delta u)_{,j}-\frac{n+2}{2n}(1+\frac{n\alpha}{n+4})\frac{\Delta u}{u}u_j+\frac 1 4(1+\frac{n\alpha}{n+4})(\frac{n+2}{n}-\frac{n-2}{n+4}\alpha)\frac{|\nabla u|^2}{u^2}u_j,\]
    \begin{align*}
        G=~&\Delta^2 u-\frac{n+2}{2n}(1+\frac{n\alpha}{n+4})\frac{(\Delta u)^2}{u}+\frac{n+2}{2n}(1+\frac{n\alpha}{n+4})(1-\frac{n\alpha}{n+4})\frac{\Delta u|\nabla u|^2}{u^2}\\
        &-\frac 1 8(1+\frac{n\alpha}{n+4})(1-\frac{3n\alpha}{n+4})(\frac{n+2}{n}-\frac{n-2}{n+4}\alpha)\frac{|\nabla u|^4}{u^3}.
    \end{align*}
    Relationships between them via differentiation are as follows:
    \[{E_{i,}}^i=u^{-1}(E_{ij}E^{ij}+R_{ij}u^iu^j)-\Big(\frac{n-1}{n}-\frac{\alpha}{n+4}\Big)E+\frac{n-1}{n} F,\]
    \begin{align*}
        {F_{i,}}^i=\frac 1 2(1+\frac{n\alpha}{n+4})(\frac{n+2}{n}-\frac{n-2}{n+4}\alpha)E-\frac{n+2}{2n}(1+\frac{n\alpha}{n+4})F+G,
    \end{align*}
    \begin{align*}
        {G_{,}}^i=~&(1+\frac{n\alpha}{n+4})\Big[-\frac 1 2(1-\frac{3n\alpha}{n+4})(\frac{n+2}{n}-\frac{n-2}{n+4}\alpha)\frac{|\nabla u|^2}{u^2}+\frac{n+2}{n}(1-\frac{n\alpha}{n+4})\frac{\Delta u}{u}\Big]E^i\\
        &+\Big[\frac{n+2}{2n}(1+\frac{n\alpha}{n+4})\Big((1-\frac{n\alpha}{n+4})\frac{|\nabla u|^2}{u^2}-2\frac{\Delta u}{u}\Big)\Big]F^i+\alpha\frac{G}{u}u^i\\
        &-\frac{\alpha}{2(n+4)}(1+\frac{n\alpha}{n+4})(1-\frac{n-4}{n+4}\alpha)\Big(n+2-\frac{n(n-2)}{n+4}\alpha\Big)\frac{|\nabla u|^4}{u^4} u^i \\
        &+\frac \alpha 2(1+\frac{n\alpha}{n+4})(1-\frac{n-4}{n+4}\alpha)\frac{\Delta u|\nabla u|^2}{u^3}u^i.
    \end{align*}
\end{proposition}

The key differential identity is established by Proposition \ref{prop:invariant}.

\begin{proposition}\label{prop:id}
    For \(c_2=\frac{(n^2+2n+4)}{(n-1)(n+4)}\alpha+\frac{n+2}{n-1}\) and
    \[c_1=-\frac{n^2(3n-10)}{4(n-1)(n+4)^2}\alpha^2+\frac{2(n+2)}{(n-1)(n+4)}\alpha+\frac{3(n+2)}{4(n-1)},\]
    the following identity holds:
    \begin{align}
        \begin{split}\label{identity}
            &u^{\frac{2\alpha}{n+4}}\Big\{u^{-\frac{2\alpha}{n+4}}\Big[\Big(c_1\frac{|\nabla u|^2}{u}-c_2 \Delta u\Big)E_i+\Big(\frac{n+2}{n+4}\alpha\frac{|\nabla u|^2}{u}+\Delta u\Big)F_i\\
            &~~~~~~~~~~~~~~~~~~~~~{-Gu_i+\frac{n\alpha}{2(n+4)}(1+\frac{n\alpha}{n+4})(1-\frac{n-4}{n+4}\alpha)\frac{|\nabla u|^4}{u^3}u_i\Big]\Big\}_{,}}^i\\
            =~&\Big(c_1\frac{|\nabla u|^2}{u^2}-c_2\frac{\Delta u}{u}\Big)(E_{ij}E^{ij}+R_{ij}u^iu^j)+2c_1E_iE^i+F_iF^i\\
            &+2A_{12}E_iF^i+2A_{13}\frac{|\nabla u|^2}{u}E+2A_{23}\frac{|\nabla u|^2}{u}F+A_{33}\frac{|\nabla u|^6}{u^4},
        \end{split}
    \end{align}
    where the coefficients are
    \[A_{12}=\frac{n^2-8}{2(n-1)(n+4)}\alpha-\frac{n+2}{2(n-1)},\quad A_{13}=\frac{\alpha}{n+4}(\frac{n^2\alpha}{n+4}+n+1)(1-\frac{n-4}{n+4}\alpha),\]
    \[A_{23}=-\frac \alpha 4(1-\frac{n-4}{n+4}\alpha),\quad A_{33}=\frac{n(n-2)}{2(n+4)^2}\alpha^2(1+\frac{n\alpha}{n+4})(1-\frac{n-4}{n+4}\alpha).\]
\end{proposition}

\begin{proof}    
    By Proposition \ref{prop:invariant}, we derive the following differential identities:
    \begin{align*}
        u^{\frac{2\alpha}{n+4}}{(u^{-\frac{2\alpha}{n+4}-1}|\nabla u|^2E_i)_{,}}^{i}=~&\frac{|\nabla u|^2}{u^2}(E_{ij}E^{ij}+R_{ij}u^iu^j)+2E_iE^i\\
        &+(\frac{n-2}{n+4}\alpha-1)\frac{|\nabla u|^2}{u}E+\frac 2 n \Delta u E+\frac{n-1}{n}\frac{|\nabla u|^2}{u}F,
    \end{align*}
    \begin{align*}
     u^{\frac{2\alpha}{n+4}}{(u^{-\frac{2\alpha}{n+4}} \Delta u E_i)_{,}}^i=~&\frac{\Delta u}{u}(E_{ij}E^{ij}+R_{ij}u^iu^j)+E_i F^i+\Big(\frac{n\alpha}{2(n+4)}-\frac{n-4}{2n}\Big)\Delta u E\\
        &-\frac 1 4(1+\frac{n\alpha}{n+4})(\frac{n+2}{n}-\frac{n-2}{n+4}\alpha)\frac{|\nabla u|^2}{u}E+\frac{n-1}{n}\Delta u F,
    \end{align*}
    \begin{align*}
     u^{\frac{2\alpha}{n+4}}{(u^{-\frac{2\alpha}{n+4}-1}|\nabla u|^2 F_i)_{,}}^i=~&2E_iF^i+\frac 1 2(1+\frac{n\alpha}{n+4})(\frac{n+2}{n}-\frac{n-2}{n+4}\alpha)\frac{|\nabla u|^2}{u}E\\
        &+\Big(\frac{n-8}{2(n+4)}\alpha-\frac{n+4}{2n}\Big)\frac{|\nabla u|^2}{u}F+\frac 2 n \Delta u F+\frac{|\nabla u|^2}{u}G,
    \end{align*}
    \begin{align*}
        u^{\frac{2\alpha}{n+4}}{(u^{-\frac{2\alpha}{n+4}} \Delta u F_i)_{,}}^i=~&F_iF^i-\frac{2\alpha}{n+4}\Delta u F+\Delta u G\\
        &+\frac 1 4(1+\frac{n\alpha}{n+4})(\frac{n+2}{n}-\frac{n-2}{n+4}\alpha)\Big(2\Delta u E-\frac{|\nabla u|^2}{u}F\Big),
    \end{align*}
    \begin{align*}
        &u^{\frac{2\alpha}{n+4}}{(u^{-\frac{2\alpha}{n+4}} Gu_i)_{,}}^i\\
        =~&(1+\frac{n\alpha}{n+4})\Big(-\frac 1 2(1-\frac{3n\alpha}{n+4})(\frac{n+2}{n}-\frac{n-2}{n+4}\alpha)\frac{|\nabla u|^2}{u}+\frac{n+2}{n}(1-\frac{n\alpha}{n+4})\Delta u\Big)E\\
        &+\frac{n+2}{2n}(1+\frac{n\alpha}{n+4})\Big((1-\frac{n\alpha}{n+4})\frac{|\nabla u|^2}{u}-2\Delta u\Big)F+\frac{n+2}{n+4}\alpha\frac{|\nabla u|^2}{u}G+\Delta uG\\
        &+\frac \alpha 2(1+\frac{n\alpha}{n+4})(1-\frac{n-4}{n+4}\alpha)\Big[\Big(\frac{n(n-2)}{(n+4)^2}\alpha-\frac{n+2}{n+4}\Big)\frac{|\nabla u|^2}{u}+\Delta u\Big]\frac{|\nabla u|^4}{u^3}.
    \end{align*}
    \[u^{\frac{2\alpha}{n+4}}{(u^{-\frac{2\alpha}{n+4}-3}|\nabla u|^4u_i)_,}^i=4\frac{|\nabla u|^2}{u}E+\Big(\frac{2(n-2)}{n+4}\alpha-\frac{n+2}{n}\Big)\frac{|\nabla u|^6}{u^4}+\frac{n+4}{n}\frac{\Delta u|\nabla u|^4}{u^3}.\]
    We finish the proof by linearly combining six identities.
\end{proof}

We prepare the following lemma for proving Theorem \ref{thm:subcritical}.

\begin{lemma}
    Under assumptions of Proposition \ref{prop:id}, when \(0<\alpha<\frac{n+4}{n-4}\) and \(n\geqslant5\), there exists constants \(\delta_1,\delta_2>0\) depending on \(n\) and \(\alpha\) such that
    \begin{align}
        \begin{split}\label{inequality}
            &\Big\{u^{-\frac{2\alpha}{n+4}}\Big[\Big(c_1\frac{|\nabla u|^2}{u}-c_2 \Delta u\Big)E_i+\frac{n+2}{n+4}\alpha\frac{|\nabla u|^2}{u}F_i+(1-\frac{n+4}{n+2}\frac{\delta_2}{\alpha})(\Delta uF_i-Gu_i)\\
            &~~~~~~~~~~~~~~~{+\frac{n\alpha}{2(n+4)}(1+\frac{n\alpha}{n+4})(1-\frac{n-4}{n+4}\alpha)\frac{|\nabla u|^4}{u^3}u_i+\delta_1\frac{\Delta u|\nabla u|^2}{u^2}u_i\Big]\Big\}_{,}}^i\\
            \geqslant~&\delta_2 u^{-\frac{2\alpha}{n+4}}\Big(E_iE^i+F_iF^i+\frac{|\nabla u|^6}{u^4}+\frac{(\Delta u)^2|\nabla u|^2}{u^2}+\frac{|\nabla u|^2}{u}\Delta^2u\Big).
        \end{split}
    \end{align}
\end{lemma}

\begin{proof}
    By nonnegativity of Ricci curvature and Proposition \ref{prop:est}, we obtain
    \[R_{ij}u^iu^j\geqslant 0,\quad-\Delta u\geqslant\frac{2}{n-4}\frac{|\nabla u|^2}{u^2}.\]
    Similar as the proof in \cite[Lemma 2.2]{MOW25}, we have \(|\nabla u|^2 E_{ij} E^{ij}\geqslant \frac 4 3 E_i E^i\). Obviously, \(c_2>0\). Thus the identity \eqref{identity} yields
    \begin{align}
        \begin{split}\label{ineq:id-1}
            &u^{\frac{2\alpha}{n+4}}\Big\{u^{-\frac{2\alpha}{n+4}}\Big[\Big(c_1\frac{|\nabla u|^2}{u}-c_2 \Delta u\Big)E_i+\Big(\frac{n+2}{n+4}\alpha\frac{|\nabla u|^2}{u}+\Delta u\Big)F_i\\
            &~~~~~~~~~~~~~~~~~~~~~{-Gu_i+\frac{n\alpha}{2(n+4)}(1+\frac{n\alpha}{n+4})(1-\frac{n-4}{n+4}\alpha)\frac{|\nabla u|^4}{u^3}u_i\Big]\Big\}_{,}}^i\\
            \geqslant~&\Big(E_i,F_i,\frac{|\nabla u|^2}{u^2}u_i\Big)A\Big(E^i,F^i,\frac{|\nabla u|^2}{u^2}u^i\Big)^T,
        \end{split}
    \end{align}
    where the symmetric matrix \(A=(A_{kl})_{3\times3}\) satifies that \(A_{22}=1\),
    \begin{align*}
        &A_{11}=\frac{n}{n-1}(c_1+\frac{2c_2}{n-4})+2c_1\\
        =~&-\frac{n^2(3n-2)(3n-10)}{4(n-1)^2(n+4)^2}\alpha^2+\frac{4(2n^3-3n^2-8n+8)}{(n-1)^2(n-4)(n+4)}\alpha+\frac{(n+2)(9n^2-34n+24)}{4(n-1)^2(n-4)},
    \end{align*}
    and \(A_{12},A_{13},A_{23},A_{33}\) are defined in Proposition \ref{prop:id}.
    
    By straightforward computations, we have
    \[\begin{vmatrix}
        A_{11} & A_{12}\\
        A_{12} & A_{22}
    \end{vmatrix}=\frac{(n-2)^2}{2(n-1)^2(n-4)^2}f_1\Big(\frac{(n-4)\alpha}{(n-2)(n+4)}\Big),\]
    \[\det A=\frac{n\alpha^2}{64(n-1)^2(n+4)^2}(\frac{1}{n-4}-\frac{\alpha}{n+4})f_2(\frac{\alpha}{n+4}),\]
    where \(f_1(x)=-(5n^4-18n^3+2n^2+32)x^2+(n^3+16n^2-8n-64)x+4(n+2)(n-4),\)
    \begin{align*}
        f_2(x)=~&-n(n-4)(9n^5-48n^4+148n^3-112n^2+448n-256)x^3\\
        &-n(7n^5-224n^4+972n^3-1936n^2+1088n+256)x^2\\
        &+(n-2)(9n^4+118n^3-288n^2+96n+512)x+(n-2)^2(n-4)(7n+32).
    \end{align*}
    When \(x\in(0,\frac{1}{n-2})\), \(n\geqslant5\),
    \[f_1(x)>\min\Big\{f_1(0),f_1(\frac{1}{n-2})\Big\}=\min\Big\{4(n+2)(n-4),\frac{8n^3-26n^2+48n-32)}{(n-2)^2}\Big\}>0.\]
    When \(x\in(0,\frac{1}{n-4})\), \(n\geqslant5\), by \(x^3<\frac{1}{n-4}x^2\), we obtain \(f_2(x)>(n-2)f_3(x)\) with
    \begin{align*}
        f_3(x)=~&-16n^2(n-12)(n^2-3n+4)x^2\\
        &+(9n^4+118n^3-288n^2+96n+512)x+(n-2)(n-4)(7n+32).
    \end{align*}
    If \(5\leqslant n\leqslant 12\), then
    \[f_3(x)\geqslant(9n^4+118n^3-288n^2+96n+512)x+(n-2)(n-4)(7n+32)>0.\]
    If \(n\geqslant 13\), then
    \begin{align*}
        f_3(x)
        >~&\min\Big\{f_3(0),f_3(\frac{1}{n-4})\Big\}\\
        =~&\min\Big\{(n-2)(n-4)(7n+32),\frac{64(n-2)^2(4n^3-13n^2+24n-16)}{(n-4)^2}\Big\}>0.
    \end{align*}
    Thus, when \(0<\alpha<\frac{n+4}{n-4}\) and \(n\geqslant5\), the matrix \(A\) is positive definite. Then \eqref{ineq:id-1} yields
    \begin{align}
        \begin{split}\label{ineq:id-2}
            &u^{\frac{2\alpha}{n+4}}\Big\{u^{-\frac{2\alpha}{n+4}}\Big[\Big(c_1\frac{|\nabla u|^2}{u}-c_2 \Delta u\Big)E_i+\Big(\frac{n+2}{n+4}\alpha\frac{|\nabla u|^2}{u}+\Delta u\Big)F_i\\
            &~~~~~~~~~~~~~~~~~~~~~{-Gu_i+\frac{n\alpha}{2(n+4)}(1+\frac{n\alpha}{n+4})(1-\frac{n-4}{n+4}\alpha)\frac{|\nabla u|^4}{u^3}u_i\Big]\Big\}_{,}}^i\\
            \geqslant~&\lambda_{\min}(A)\Big(E_iE^i+F_iF^i+\frac{|\nabla u|^6}{u^4}\Big),
        \end{split}
    \end{align}
    where \(\lambda_{\min}(A)\), depending on \(n\) and \(\alpha\), is the minimal eigenvalue of \(A\).

    By Proposition \ref{prop:invariant} and Cauchy's inequality, we obtain
    \begin{align*}
        u^{\frac{2\alpha}{n+4}}{[u^{-\frac{2\alpha}{n+4}-2}\Delta u|\nabla u|^2u_i]_,}^i=~&\frac{n+2}{n}\frac{(\Delta u)^2|\nabla u|^2}{u^2}+\frac 1 2(\frac{3n-4}{n+4}\alpha-1)\frac{\Delta u|\nabla u|^4}{u^3}\\
        &+\frac 1 4(1+\frac{n\alpha}{n+4})(\frac{n-2}{n+4}\alpha-\frac{n+2}{n})\frac{|\nabla u|^6}{u^4}+2\Delta uE+\frac{|\nabla u|^2}{u}F\\
        \geqslant~&\frac{(\Delta u)^2|\nabla u|^2}{u^2}-C\Big(E_iE^i+F_iF^i+\frac{|\nabla u|^6}{u^4}\Big).
    \end{align*}
    By the proof of Proposition \ref{prop:id}, the definition of \(G\), and Cauchy's inequality, we obtain
    \begin{align*}
        &u^{\frac{2\alpha}{n+4}}{[u^{-\frac{2\alpha}{n+4}}(-\Delta u F_i+Gu_i)]_{,}}^i\\
        \geqslant~&\frac{n+2}{n+4}\alpha\frac{|\nabla u|^2}{u}G-C\Big(E_iE^i+F_iF^i+\frac{|\nabla u|^6}{u^4}+\frac{(\Delta u)^2|\nabla u|^2}{u^2}\Big)\\
        \geqslant~&\frac{n+2}{n+4}\alpha\frac{|\nabla u|^2}{u}\Delta^2u-C\Big(E_iE^i+F_iF^i+\frac{|\nabla u|^6}{u^4}+\frac{(\Delta u)^2|\nabla u|^2}{u^2}\Big)
    \end{align*}
    Combining with \eqref{ineq:id-2} by choosing \(\delta_1>0\) and \(\frac{\delta_2}{\delta_1}>0\) small enough, \eqref{inequality} holds.
\end{proof}
\section{Integral estimates: proof of Theorem \ref{thm:subcritical}}\label{sec:subcritical}

In this section, thanks to preparations in Section \ref{sec:identity}, we prove Theorem \ref{thm:subcritical} by establishing some integral estimates. Denote \(\varepsilon\) as a sufficiently small positive constant independent of \(R\), and all other symbols are the same as in Section \ref{sec:identity}.

\begin{lemma}\label{lem:int-1}
    For any \(\gamma\geqslant6\), there exists a constant \(C>0\) depending on \(n,\alpha,\gamma\) such that
    \begin{align*}
        R^{-2}\int_M u^{-\frac{2\alpha}{n+4}+1}\Delta^2u\eta^{\gamma-2}\leqslant~&\varepsilon\int_M u^{-\frac{2\alpha}{n+4}}\Big(E_iE^i+F_iF^i+\frac{|\nabla u|^6}{u^4}\Big)\eta^\gamma\\
        &+CR^{-6}\int_M u^{-\frac{2\alpha}{n+4}+2}\eta^{\gamma-6}+CR^{-2}\int_M u^{-\frac{2\alpha}{n+4}}(\Delta u)^2\eta^{\gamma-2}.
    \end{align*}
\end{lemma}

\begin{proof}
    By Proposition \ref{prop:invariant}, the definition of \(G\), and Cauchy's inequality, we obtain
    \begin{align*}
        u^{\frac{2\alpha}{n+4}}{(u^{-\frac{2\alpha}{n+4}+1}F_i)_,}^i=~&uG+\frac 1 2(1+\frac{n\alpha}{n+4})(\frac{n+2}{n}-\frac{n-2}{n+4}\alpha)uE-\frac{n+2}{2n}(1+\frac{n\alpha}{n+4})uF\\
        \geqslant~&u\Delta^2 u-C\Big((|E_i|+|F_i|)|u^i|+\frac{|\nabla u|^4}{u^2}+(\Delta u)^2\Big).
    \end{align*}
    By testing \(u^{-\frac{2\alpha}{n+4}}\eta^{\gamma-2}\) on it, and applying Young's inequality, we obtain
    \begin{align*}
        R^{-2}\int_M u^{-\frac{2\alpha}{n+4}+1}\Delta^2u\eta^{\gamma-2}\leqslant~&CR^{-2}\int_M u^{-\frac{2\alpha}{n+4}}\Big((|E_i|+|F_i|)|u^i|+\frac{|\nabla u|^4}{u^2}+(\Delta u)^2\Big)\eta^{\gamma-2}\\
        &+CR^{-2}\int_M u^{-\frac{2\alpha}{n+4}+1}|F^i|\eta^{\gamma-3}|\eta_{,i}|.
    \end{align*}
    Thus, we complete the proof by Young's inequality.
\end{proof}

\begin{lemma}\label{lem:int-2}
    For any \(\gamma\geqslant6\), there exists a constant \(C>0\) depending on \(n,\alpha,\gamma\) such that
    \[R^{-2}\int_M u^{-\frac{2\alpha}{n+4}}(\Delta u)^2\eta^{\gamma-2}\leqslant\varepsilon\int_M u^{-\frac{2\alpha}{n+4}}\Big(F_iF^i+\frac{|\nabla u|^6}{u^4}\Big)\eta^\gamma+CR^{-6}\int_M u^{-\frac{2\alpha}{n+4}+2}\eta^{\gamma-6}.\]
\end{lemma}

\begin{proof}
    By Proposition \ref{prop:invariant} and Cauchy's inequality, we obtain
    \begin{align*}
        u^{\frac{2\alpha}{n+4}}{(u^{-\frac{2\alpha}{n+4}}\Delta uu_i)_,}^i=~&(\Delta u)^2+uF+\frac 1 2(\frac{n-2}{n+4}\alpha+\frac{n+2}{n})\frac{|\nabla u|^2}{u}\Big(\Delta u+\frac 1 3(1+\frac{n\alpha}{n+4})\frac{|\nabla u|^2}{u}\Big)\\
        \geqslant~&\frac 1 2(\Delta u)^2-C\Big(|F_i||u^i|+\frac{|\nabla u|^4}{u^2}\Big).
    \end{align*}
    By testing \(u^{-\frac{2\alpha}{n+4}}\eta^{\gamma-2}\) on it, and applying Young's inequality, we obtain
    \begin{align*}
        &R^{-2}\int_M u^{-\frac{2\alpha}{n+4}}(\Delta u)^2\eta^{\gamma-2}\\
        \leqslant~&CR^{-2}\int_M u^{-\frac{2\alpha}{n+4}}\Big(|F_i||u^i|+\frac{|\nabla u|^4}{u^2}\Big)\eta^{\gamma-2}+CR^{-3}\int_M u^{-\frac{2\alpha}{n+4}}\Delta u|\nabla u|\eta^{\gamma-3}\\
        \leqslant~&\frac 1 2R^{-2}\int_M u^{-\frac{2\alpha}{n+4}}(\Delta u)^2\eta^{\gamma-2}+\varepsilon\int_M u^{-\frac{2\alpha}{n+4}}\Big(F_iF^i+\frac{|\nabla u|^6}{u^4}\Big)\eta^\gamma+CR^{-6}\int_M u^{-\frac{2\alpha}{n+4}+2}\eta^{\gamma-6}.
    \end{align*}
    Thus, we complete the proof.
\end{proof}

\begin{proof}[\normalfont\bfseries Proof of Theorem \ref{thm:subcritical}]
    Let \(\gamma\geqslant6\) be a constant to be determined. Testing \(\eta^\gamma\) on \eqref{inequality} yields
    \begin{align}
        \begin{split}\label{int:1}
            &\int_M u^{-\frac{2\alpha}{n+4}}\Big(E_iE^i+F_iF^i+\frac{|\nabla u|^6}{u^4}+\frac{(\Delta u)^2|\nabla u|^2}{u^2}+\frac{|\nabla u|^2}{u}\Delta^2u\Big)\eta^\gamma\\
            \leqslant~&C\int_M u^{-\frac{2\alpha}{n+4}}\Big[\Big(\frac{|\nabla u|^2}{u}-\Delta u\Big)(|E^i|+|F^i|)+\Big(|G|+\frac{|\nabla u|^4}{u^3}+\frac{\Delta u|\nabla u|^2}{u^2}\Big)|u^i|\Big]\eta^{\gamma-1}|\eta_{,i}|.
        \end{split}
    \end{align}
    
    By Young's inequality and the definition of \(G\), we obtain
    \begin{align*}
        &\int_M u^{-\frac{2\alpha}{n+4}}\Big(\frac{|\nabla u|^2}{u}-\Delta u\Big)(|E^i|+|F^i|)\eta^{\gamma-1}|\eta_{,i}|\\
        \leqslant~&\varepsilon\int_M u^{-\frac{2\alpha}{n+4}}(E_iE^i+F_iF^i)\eta^\gamma+CR^{-2}\int_M u^{-\frac{2\alpha}{n+4}}\Big(\frac{|\nabla u|^4}{u^2}+(\Delta u)^2\Big)\eta^{\gamma-2}\\
        \leqslant~&\varepsilon\int_M u^{-\frac{2\alpha}{n+4}}\Big(E_iE^i+F_iF^i+\frac{|\nabla u|^6}{u^4}\Big)\eta^\gamma\\
        &+CR^{-6}\int_M u^{-\frac{2\alpha}{n+4}+2}\eta^{\gamma-6}+CR^{-2}\int_M u^{-\frac{2\alpha}{n+4}}(\Delta u)^2\eta^{\gamma-2},
    \end{align*}
    \begin{align*}
        &\int_M u^{-\frac{2\alpha}{n+4}}\Big(|G|+\frac{|\nabla u|^4}{u^3}+\frac{\Delta u|\nabla u|^2}{u^2}\Big)|u^i|\eta^{\gamma-1}|\eta_{,i}|\\
        \leqslant~&\varepsilon\int_M u^{-\frac{2\alpha}{n+4}}\Big(\frac{|\nabla u|^6}{u^4}+\frac{(\Delta u)^2|\nabla u|^2}{u^2}+\frac{|\nabla u|^2}{u}\Delta^2u\Big)\eta^\gamma\\
        &+CR^{-6}\int_M u^{-\frac{2\alpha}{n+4}+2}\eta^{\gamma-6}+CR^{-2}\int_M u^{-\frac{2\alpha}{n+4}+1}\Delta^2u\eta^{\gamma-2}.
    \end{align*}
    Inserting them into \eqref{int:1}, we have
    \begin{align*}
        &\int_M u^{-\frac{2\alpha}{n+4}}\Big(E_iE^i+F_iF^i+\frac{|\nabla u|^6}{u^4}+\frac{(\Delta u)^2|\nabla u|^2}{u^2}+\frac{|\nabla u|^2}{u}\Delta^2u\Big)\eta^\gamma\\
        \leqslant~&CR^{-6}\int_M u^{-\frac{2\alpha}{n+4}+2}\eta^{\gamma-6}+CR^{-2}\int_M u^{-\frac{2\alpha}{n+4}}(\Delta u)^2\eta^{\gamma-2}+CR^{-2}\int_M u^{-\frac{2\alpha}{n+4}+1}\Delta^2u\eta^{\gamma-2}.
    \end{align*}
    Combining Lemma \ref{lem:int-1}, \ref{lem:int-2}, and the above inequality together, it yields
    \begin{equation}\label{int:2}
        \int_M u^{-\frac{2\alpha}{n+4}}\Big(E_iE^i+F_iF^i+\frac{|\nabla u|^6}{u^4}\Big)\eta^\gamma\leqslant CR^{-6}\int_M u^{-\frac{2\alpha}{n+4}+2}\eta^{\gamma-6}.
    \end{equation}
    By Lemma \ref{lem:int-1}, \ref{lem:int-2}, and inequality \eqref{int:2}, we obtain
    \begin{equation}\label{int:3}
        R^{-2}\int_M u^{-\frac{2\alpha}{n+4}+1}\Delta^2u\eta^{\gamma-2}+\int_M u^{-\frac{2\alpha}{n+4}-4}|\nabla u|^6\eta^\gamma\leqslant CR^{-6}\int_M u^{-\frac{2\alpha}{n+4}+2}\eta^{\gamma-6}.
    \end{equation}
    
    When \(1<\alpha<\frac{n+4}{n-4}\), choose \(\gamma=\max\Big\{6,\frac{1}{\alpha-1}(\frac{6n+16}{n+4}\alpha-2)\Big\}\), by Young's inequality,
    \[R^{-6}\int_M u^{-\frac{2\alpha}{n+4}+2}\eta^{\gamma-6}\leqslant\varepsilon R^{-2}\int_M u^{\frac{n+2}{n+4}\alpha+1}\eta^{\gamma-2}+CR^{-\frac{1}{\alpha-1}(\frac{6n+16}{n+4}\alpha+2)}\operatorname{Vol}(B_{2R}).\]
    By Bishop-Gromov comparison theorem, \(\operatorname{Vol}(B_{2R})\leqslant CR^n\), thus \eqref{int:3} yields
    \[\int_M u^{-\frac{2\alpha}{n+4}-4}|\nabla u|^6\eta^\gamma\leqslant CR^{n-\frac{1}{\alpha-1}(\frac{6n+16}{n+4}\alpha+2)}=CR^{\frac{(n^2-2n-16)\alpha-(n+2)(n+4)}{(n+4)(\alpha-1)}}.\]
    Notice that \(\frac{(n^2-2n-16)\alpha-(n+2)(n+4)}{(n+4)(\alpha-1)}<\frac{-8}{(n-4)(\alpha-1)}<0\). By letting \(R\to\infty\), we obtain that \(u\) is zero, which contradicts with \(u>0\). 
\end{proof}

\textbf{Acknowledgments.} Xi-Nan Ma was supported by the National Natural Science Foundation of China (Grant No. 12141105) and the National Key Research and Development Project (Grant No. SQ2020YFA070080). Tian Wu was supported by Anhui Postdoctoral Scientific Research Program Foundation (Grant No. 2025B1055).

%\printbibliography[heading=bibintoc, title=\ebibname]
%\appendix

\footnotesize{
    Welcome contact us:
    \begin{itemize}
        \item Xi-Nan Ma, School of Mathematical Sciences, University of Science and Technology of China, Hefei, Anhui, 230026, People's Republic of China. Email: \emph{xinan@ustc.edu.cn}
        \item Tian Wu, School of Mathematical Sciences, University of Science and Technology of China, Hefei, Anhui, 230026, People's Republic of China. Email: \emph{wt1997@ustc.edu.cn}
        \item Wangzhe Wu, Institute of Mathematics, Academy of Mathematics and Systems Science, Chinese Academy of Sciences, Beijing, 100190, People's Republic of China. Email: \emph{wuwz18@mail.ustc.edu.cn}
    \end{itemize}
}

\end{document}